\documentclass[11pt]{article}

\usepackage{amsfonts}
\usepackage{amssymb,amsmath,amsthm}
\usepackage{latexsym}
\usepackage{fullpage}
\usepackage{hyperref}
\usepackage{graphicx}
\usepackage{tkz-euclide,subfigure}
\usepackage{tikz}
\usepackage[utf8]{inputenc}
\usepackage{pgfplots}
\usetikzlibrary{shapes.misc}
\usepackage{ amssymb }

\newtheorem{theorem}{Theorem}[section]

\newtheorem{lemma}[theorem]{Lemma}

\theoremstyle{definition}
\newtheorem{remark}[theorem]{Remark}

\newcounter{tenumerate}

\def\P{\mathbb{P}}

\renewcommand{\epsilon}{\varepsilon}

\newcommand{\E}{{\mathbb E}}
\newcommand{\remove}[1]{}

\renewcommand{\leq}{\leqslant}
\renewcommand{\geq}{\geqslant}

\def\XXint#1#2#3{{\setbox0=\hbox{$#1{#2#3}{\int}$}
\vcenter{\hbox{$#2#3$}}\kern-.5\wd0}}

\title{Long range order for random field Ising and Potts models}
\author{Jian Ding\thanks{Partially supported by NSF grant DMS-1953848.}  \\ University of Pennsylvania \and Zijie Zhuang\footnotemark[1]  \\ University of Pennsylvania}

\begin{document}

\maketitle

\begin{abstract}
We present a new and simple proof for the classic results of Imbrie (1985) and Bricmont--Kupiainen (1988) that for the random field Ising model in dimension three and above there is long range order at low temperatures with presence of weak disorder. With the same method, we obtain a couple of new results: (1) we prove that long range order exists for the random field Potts model at low temperatures with presence of weak disorder in dimension three and above; (2) we obtain a lower bound on the correlation length for the random field Ising model at low temperatures in dimension two (which matches the upper bound in Ding--Wirth (2020)).
Our proof is based on an extension of the Peierls argument with inputs from Chalker (1983), Fisher--Fr\"ohlich--Spencer (1984), Ding--Wirth (2020) and Talagrand's majorizing measure theory (1980s) (and in particular, our proof does not involve the renormalization group theory).
\end{abstract}

\section{Introduction}

For $d\geq 2$, we consider the $d$-dimensional lattice $\mathbb Z^d$ where two vertices $u$ and $v$ are adjacent (and we write $u\sim v$) if their $\ell_1$-distance is 1.
For $N\geq 1$, let $\Lambda_N=[-N,N]^d \cap \mathbb Z^d$ be the box of side length $2N$ centered at the origin $o$. For $v\in \mathbb Z^d$, let $h_v$ be i.i.d.\ Gaussian random variables with mean 0 and variance 1 (we denote by $\mathbb P$ and $\mathbb E$ the measure and expectation with respect to $\{h_v\}$, respectively). For $\epsilon \geq 0$, the random field Ising model (RFIM) Hamiltonian $H^{\pm, \Lambda_N, \epsilon h}$ with the plus (respectively minus) boundary condition and external field $\{ \epsilon h_v :v \in \mathbb{Z}^d \}$ is defined to be
\begin{equation}
\label{def_h}
H^{\pm, \Lambda_N, \epsilon h} (\sigma)= - \Big(\sum_{u \sim v, u,v \in \Lambda_N} \sigma_u \sigma_v \pm \sum_{u \sim v, u \in \Lambda_N, v\in \Lambda_N^c} \sigma_u + \sum_{u \in \Lambda_N} \epsilon h_u \sigma_u\big),
\end{equation}
where $\sigma \in \{-1, 1\}^{\Lambda_N}$.
For $T \geq 0$, we define $\mu^{\pm}_{T, \Lambda_N, \epsilon h} $ to be the Gibbs measure on $\{-1,1\}^{\Lambda_N}$ at temperature $T$ by
\begin{equation}
\label{def_mu}
\mu^{\pm}_{T, \Lambda_N, \epsilon h} (\sigma)= \frac{1}{\mathcal{Z}^\pm_{T, \Lambda_N} (\epsilon h)} e^{-\frac{1}{T} H^{\pm, \Lambda_N, \epsilon h} (\sigma)},
\end{equation}
where (again) $\sigma\in \{-1, 1\}^{\Lambda_N}$ and $\mathcal{Z}^\pm_{T, \Lambda_N}( \epsilon h)$ is the partition function given by
\begin{equation}
\label{def_z}
\mathcal{Z}^\pm_{T, \Lambda_N}( \epsilon h) = \sum_{\sigma \in \{-1,1\}^{\Lambda_N}} e^{-\frac{1}{T} H^{\pm, \Lambda_N, \epsilon h} (\sigma)}.
\end{equation}
(Note that $\mu^\pm_{T, \Lambda_N, \epsilon h}$ and $\mathcal Z^\pm_{T, \Lambda_N} (\epsilon h)$ are random variables depending on $\{h_v\}$.)
We say long range order exists for RFIM if for a typical instance of the disorder, the boundary influence $m_{T,\Lambda_N, \epsilon h}$ stays above a positive constant as $N\to \infty$, where the boundary influence is defined as
$$
m_{T,\Lambda_N, \epsilon h} = \mu^{+}_{T, \Lambda_N, \epsilon h}(\sigma_o = 1) - \mu^-_{T, \Lambda_N, \epsilon h}(\sigma_o = 1)\,.
$$
The question on long range order is relatively easy when the disorder is strong, i.e., when $\epsilon$ is large. In this case, it was shown in \cite{Ber85, FI84, CJN18} (see also \cite[Appendix A]{AP19}) that for any dimension, the boundary influence decays exponentially in $N$ (so in particular no long range order). The question becomes substantially more challenging when the disorder is weak, i.e., when $\epsilon$ is small. In this case, it was predicted in \cite{IM75} that long range order exists at low temperatures for $d\geq 3$ but not for $d=2$. There has been controversy over this prediction for quite some time, and it was finally proved to be correct by \cite{Imbrie85, BK88} for $d=3$ and by \cite{AW90} for $d=2$. In recent works  \cite{Chatterjee18, AP19, DX21, AHP20} quantitative bounds on the decay rate in dimension two were obtained, and in particular, exponential decay was finally established in \cite{DX21, AHP20}. In addition, the authors of \cite{DW20} studied (a notion of) the correlation length, defined as
\begin{equation}
\label{def_psi}
\psi(T, \mathsf m, \epsilon) =\min \{N:  \mathbb E(m_{T,\Lambda_N, \epsilon h} ) \leq \mathsf m\}, \mbox{ for }\mathsf m\in (0, 1).
\end{equation}
It was shown in \cite{DW20} that for any fixed $\mathsf m\in (0, 1)$, the correlation length $\psi(T, \mathsf m, \epsilon)$ scales as $e^{\epsilon^{-4/3 + o(1)}}$ as $\epsilon \to 0$ for $T = 0$ (and the upper bound holds for all $T \geq 0$).  Finally, it is worth mentioning that in \cite{DSS21} a general inequality for the Ising model was derived which in particular implies exponential decay for RFIM as long as $T$ exceeds the critical temperature (here the critical temperature is with respect to the Ising model with no external field).

The method in the classic papers \cite{Imbrie85, BK88} for proving lower bounds on the boundary influence is based on a sophisticated scheme of renormalization group theory, and as a result it seems very difficult (if possible at all) to extend to e.g.\ the random field Potts model. All the recent works focused on proving upper bounds on the boundary influence except that a lower bound was derived in \cite{DW20} for $T = 0$. In the setting of \cite{DW20} for $T=0$ the RFIM measure is supported on either the all-plus or the all-minus configuration depending on the boundary condition, and thus the lower bound follows from an upper bound on the supremum of a Gaussian process (i.e., the greedy lattice animal process normalized by its boundary size as in \cite{DW20}). It is expected that the lower bound in \cite{DW20} holds also at low temperatures, but in attempts to extending from $T = 0$ to $T>0$ difficulties arise from hierarchical structure of sign clusters. Despite the fact that the renormalization group theoretic approach in \cite{Imbrie85, BK88} was designed exactly to tackle this challenge, it seems difficult to adapt their method from dimension three to dimension two.

In this paper, we present a simple proof for the following result of \cite{Imbrie85} ($T=0$) and of \cite{BK88} ($T>0$) on the existence of long range order.
\begin{theorem}\cite{Imbrie85, BK88}
\label{RFIM3}
For $d \geq 3$, there exists a constant $c>0$ such that for all $0\leq T, \epsilon \leq c $ and all $N\geq 1$ with $\mathbb P$-probability at least $1 - e^{-c/T} - e^{-c/\epsilon^2}$ we have
\begin{equation*}
\mu_{T, \Lambda_N, \epsilon h}^+(\sigma_o=1) \geq 1- e^{-c/T} - e^{-c/\epsilon^2}.
\end{equation*}
\end{theorem}

Our method is robust and can be applied to deduce a lower bound on the correlation length in dimension two for small $T>0$.
\begin{theorem}
\label{RFIM2}
For $d=2$ and every $\mathsf m \in (0,1)$, there exists a constant $c  =c(\mathsf m)>0$ such that for $0\leq T, \epsilon \leq c$
$$
 \psi(T, \mathsf m, \epsilon) \geq e^{\frac{c\epsilon^{-4/3}}{ \log(1/\epsilon)}}.
$$
\end{theorem}
Note that the lower bound in Theorem~\ref{RFIM2} matches the upper bound of $e^{ O(\epsilon^{-4/3})}$ as in \cite{DW20} in terms of the $4/3$-exponent.

Finally, we consider the random field Potts model (RFPM). For $\mathsf q\geq 3$, let $h_{\mathsf k, v}$ for $\mathsf k \in \{\mathsf 1, \ldots, \mathsf q\}$ and $v\in \mathbb Z^d$ be i.i.d.\ Gaussian variables with mean 0 and variance 1. (Here we used mathsf font $\mathsf 1, \ldots, \mathsf q$ to emphasize that they denote states of spins in the Potts model, but they are also numbers in nature and thus we have e.g. $\mathsf 1 + \mathsf 3 = \mathsf 4$.) The RFPM Hamiltonian with boundary condition $\mathsf k$ and external field $\epsilon h$ is defined by
\begin{equation}
\label{def_h_potts}
H^{\mathsf k, \Lambda_N, \epsilon h}(\sigma) = - \Big(\sum_{u \sim v, u,v \in \Lambda_N} \mathbf{1}_{\sigma_u=\sigma_v}  + \sum_{u \sim v, u \in \Lambda_N, v \in \Lambda_N^c} \mathbf{1}_{\sigma_u= \mathsf k}  + \sum_{u \in \Lambda_N} \epsilon h_{\sigma_u, u}\Big)\,,
\end{equation}
where $\sigma\in \{\mathsf 1, \ldots, \mathsf q\}^{\Lambda_N}$. Then, similarly, we can define
\begin{equation}
\label{def_mu_potts}
\mu^{\mathsf k}_{T, \Lambda_N, \epsilon h} (\sigma)= \frac{1}{\mathcal{Z}^{\mathsf k}_{T, \Lambda_N}(\epsilon h)} e^{- \frac{1}{T} H^{\mathsf k, \Lambda_N, \epsilon h} (\sigma)},
\end{equation}
where $\mathcal{Z}^{\mathsf k}_ {T, \Lambda_N}(\epsilon h)$ is the partition function given by
\begin{equation}
\label{def_z_potts}
\mathcal{Z}^{\mathsf k}_{T, \Lambda_N} (\epsilon h)=\sum_{\sigma \in \{\mathsf 1,\ldots, \mathsf q\}^{\Lambda_N}} e^{-\frac{1}{T} H^{\mathsf k, \Lambda_N, \epsilon h} (\sigma)}\,.
\end{equation}
It was shown in \cite{AW90} that long range order does not exist for RFPM for $d=2$ and some quantitative bound on the decay rate was obtained in \cite{DHP21}.
We prove that for $d\geq 3$, long range order exists for RFPM at low temperatures with weak disorder.
\begin{theorem}
\label{RFPM3}
For $d, \mathsf q \geq 3$, there exists a constant $c>0$ such that for all $0\leq T, \epsilon \leq c$ and for all $N\geq 1$ with $\mathbb P$-probability at least $1 - e^{-c/T} - e^{-c/\epsilon^2}$
$$
\mu^{\mathsf 1}_{T, \Lambda_N, \epsilon h}(\sigma_o= \mathsf 1) \geq 1-  e^{-c/T} -  e^{-c/ \epsilon^2}.
$$
\end{theorem}

\begin{remark}
In our presentation, we have assumed that the distribution for the disorder is Gaussian for simplicity.
We expect that our proof extends to general distribution which is symmetric around 0 and has sub-Gaussian tail, but more care is required in order to replace the application of the Gaussian concentration inequality.
\end{remark}

\section{Overview of the proof}

Our proof is based on an extension of the Peierls argument \cite{Peierls1936}. Recalling the Peierls argument, one constructs a mapping which flips every spin that is in or enclosed by the sign component (a sign component is a connected component where all spins have the same sign) at the origin (so in particular one flips a simply connected component). Then one can upper-bound the probability for the origin to disagree with the boundary condition by analyzing the probability change (of the two configurations before and after the flipping) and the multiplicity of this mapping. When applying the Peierls argument to the RFIM in a straightforward manner with external field quenched, we run into the difficulty that the probability change depends on the external field. What is worse, such probability change also depends on the configuration we start with and as a result a uniform bound is too loose in most cases. Naturally, one may wish to extend the Peierls argument. Along this line, a very insightful attempt appeared in \cite{FFS84} although it seems difficult to verify the assumed hypothesis in \cite{FFS84}.
Our modification is to apply the Peierls argument to the joint space of the external field and the spin configurations. More precisely, when we flip a spin we simultaneously flip the sign of the external field thereon. This way, there is no change in Hamiltonian resulted from the interaction with the external field. However, there is a change in the partition function (since the external field is changed) and a key input is to upper-bound the change of partition function resulted from such flipping by the boundary size of the flipped simply connected component.

In order to upper-bound the change in the partition function, we consider the free energy for RFIM and RFPM respectively defined as:
\begin{equation}\label{eq-def-free-energy}
\mathcal{F}^\pm_{T, \Lambda_N} (\epsilon h) = - T \log \mathcal{Z}^\pm_{T, \Lambda_N} (\epsilon h) \mbox{ and } \mathcal{F}^{\mathsf k}_{T, \Lambda_N}( \epsilon h) = -T \log \mathcal{Z}^{\mathsf k}_{T, \Lambda_N}(\epsilon h) \mbox{ for } \mathsf k = \mathsf 1, \ldots, \mathsf q\,.
\end{equation}
The bound on the partition function is then reduced to a bound on the free energy. By Talagrand's majorizing measure theory \cite{Talagrand87, Talagrand96, Talagrand05} (improving previous works in \cite{Dudley67, Fernique71}), we know that the supremum of a mean-zero Gaussian process is up to a constant factor within the so-called $\gamma_2$-functional and in addition the supremum of a general mean-zero random process is upper-bounded by the $\gamma_2$-functional as long as the two-point fluctuation has a sub-Gaussian tail. Altogether, we get the following lemma, which is a consequence of \cite[Theorems 2.5, 5.1]{Talagrand96} and will be particularly important in our proof of Theorem~\ref{RFIM2}.
\begin{lemma}\label{lem-Talagrand}
Suppose that $\{X_t: t\in \Omega\}$ is a mean-zero Gaussian process and suppose that $\{Y_t: t\in \Omega\}$ is a mean-zero random process such that
$$\E(X_s - X_t)^2 = (d(s, t))^2 \mbox{ and } \mathbb P(|Y_s - Y_t| \geq \lambda) \leq 2 e^{-\frac{\lambda^2}{2 (d(s, t))^2}}$$
for all $s, t\in \Omega$ and for all $\lambda \geq 0$. Then there exists a universal constant $C>0$ such that $\E \sup_{t\in \Omega} Y_t \leq C \E \sup_{t\in \Omega} X_t$.
\end{lemma}
Although Lemma~\ref{lem-Talagrand} is only important as a mathematical ingredient in the proof of Theorem~\ref{RFIM2}, the insight from Lemma~\ref{lem-Talagrand} is important for building our mental picture. This is because knowing this lemma we can then hope to bound the supremum of the change in free energy when flipping the disorder in a set (over all simply connected subsets containing the origin) by the supremum of a corresponding Gaussian process. With some further computation (see Lemma~\ref{lem-upper-bound-Delta} below), we then see that it can be bounded by the greedy lattice animal normalized by the boundary size (i.e., the supremum of the sum of the Gaussian disorder normalized by the boundary size over all simply connected subsets containing the origin) as studied in \cite{Chalker83, FFS84, DW20}. In \cite{Chalker83, FFS84}, it was shown that for $d\geq 3$ typically the greedy lattice animal normalized by the boundary size is $O(1)$ and the authors regarded this as the evidence for the existence of long range order. In some sense, our work will illustrate that this is indeed the fundamental reason behind the long range order. In contrast, it was shown in \cite{DW20} that for $d=2$ the greedy lattice animal normalized by the boundary size has a poly-log growth which governs the scaling of the correlation length. Our proof will also provide a deeper insight on its connection to the RFIM. In summary, the behavior for the greedy lattice animal normalized by the boundary size has a transition from $d=2$ to $d\geq 3$ and this offers a more accurate explanation for the RFIM transition predicted by Imry--Ma \cite{IM75}.

Finally, we remark that we will apply the same method to the random field Potts model, except that we rotate the spins and the external field instead of flipping their signs.

\section{Long range order in random field Ising model}\label{sec:RFIM}

We refer $\mathbb{P}$ and $\mathbb E$ to the probability measure and the expectation with respect to the external field $\{h_v\}$, and
we refer $\langle \cdot \rangle_{\mu^{\pm}_{T, \Lambda_N , \epsilon h }}$ to the average over the quenched measure $\mu^{\pm}_{T, \Lambda_N, \epsilon h}$. In addition,
we define the joint measure for $(h, \sigma)$ by $$\mathbb{Q}_{T, \Lambda_N, \epsilon}^{\pm}(h\in A, \sigma\in B) = \int_A \mu^{\pm}_{T, \Lambda_N, \epsilon h}(B) d\mathbb P(h)$$ for $A\subset {\mathbb R}^{\Lambda_N}$ and $B\subset \{-1, 1\}^{\Lambda_N}$ (note that this is different from the annealed measure on $\mathbb R^{\Lambda_N} \times \{-1, 1\}^{\Lambda_N}$ with density proportional to $\mathbb P(\{h_v\}) \cdot e^{-\frac{1}{T} H^{\pm, \Lambda_N, \epsilon h}(\sigma)}$). Also, for $A\subset \mathbb Z^d$, we let $\partial A = \{(u, v): u\sim v, u\in A, v\in A^c\}$ be the edge boundary of $A$. For a finite set $A$, we denote by $|A|$ the cardinality of $A$.
Since $T, \epsilon, N$ will be fixed throughout the proof, we will drop them from the superscripts/subscripts for notation convenience.

\subsection{Upper tail on free energy difference}
For $A\subset \mathbb Z^d$, we define
$$h^A_v = \begin{cases}
h_v, \mbox{ if } v\in A^c,\\
-h_v, \mbox{ if } v\in A;
\end{cases}
\mbox{ and }
\sigma^A_v =  \begin{cases}
\sigma_v, \mbox{ if } v\in A^c,\\
-\sigma_v, \mbox{ if } v\in A\,.
\end{cases}
$$  We further define  the free energy difference by
\begin{equation}\label{eq-def-Delta}
\Delta_A(h) = \mathcal F^+(h) - \mathcal F^+(h^A) \overset{\eqref{eq-def-free-energy}}{=} - T \log \mathcal Z^+(h) + T \log \mathcal Z^+(h^A)\,.
\end{equation}
The following lemma gives an upper bound on $\Delta_A(h)$.
\begin{lemma}\label{lem-upper-bound-Delta}
For any $A, A'\subset \Lambda_N$ and $\lambda > 0$, we have
\begin{align}
\P\big(|\Delta_A(h)| \geq \lambda \mid \{h_v: v\in A^c\}\big) &\leq 2 e^{-\frac{\lambda^2}{8 \epsilon^2 |A|}}\,, \label{eq-Delta-one-point}\\
\mathbb P\big(|\Delta_A(h) - \Delta_{A'}(h)| \geq \lambda \mid \{h_v: v\in (A \cup A')^c\}\big) &\leq 2 e^{-\frac{\lambda^2}{8 \epsilon^2 |A\oplus A'|}}\,, \label{eq-Delta-two-point}
\end{align}
where $A \oplus A'$ is the symmetric difference between $A$ and $A'$.
\end{lemma}
\begin{proof}
We first prove \eqref{eq-Delta-one-point}.
By symmetry of Gaussian random variables,
\begin{equation}\label{eq-Delta-expectation}
\mathbb{E} (\Delta_A(h)  \mid \{h_v: v\in A^c\}) =0\,.
\end{equation}
Furthermore, for any $v\in A$, by \eqref{def_z} and \eqref{def_h}
\begin{align}\label{eq-Delta-Lip}
\Big| \frac{\partial}{\partial_{h_{v}}} \Delta_A(h) \Big|&= \Big|-\frac{\sum_{\sigma } \epsilon   \sigma_{v} e^{-\frac{1}{T} H^{+,  \Lambda_N, \epsilon h}(\sigma)}}{\mathcal{Z}^+(h)} -  \frac{\sum_{\sigma }  \epsilon  \sigma_{v} e^{-\frac{1}{T}  H^{+,  \Lambda_N, \epsilon h^A}(\sigma)} }{\mathcal{Z}^+(h^A)} \Big| \nonumber\\
&=\epsilon \Big|   \langle \sigma_{v} \rangle_{\mu^+_{T, \Lambda_N, \epsilon h}}+\langle \sigma_{v} \rangle_{\mu^+_{T, \Lambda_N, \epsilon h^A}}   \Big| \leq  2 \epsilon.
\end{align}
Now, by the Gaussian concentration inequality (see \cite{Borell75, SudakovTsirelson74}, and see also \cite[Theorem 2.1]{Adler90} and \cite[Theorem 3.25]{Van16}), we obtain \eqref{eq-Delta-one-point}.

We next prove \eqref{eq-Delta-two-point}. In the rest of the proof, we consider probability measures conditioned on $\{h_v: v\in (A\cup A')^c\}$.
Since $\Delta_A(h) - \Delta_{A'}(h) = \mathcal F^+(h^{A'}) - \mathcal F^+(h^{A})$ and since $(h^A, h^{A'})$ has the same conditional distribution as $(h^{A\oplus A'}, h)$, we obtain that
$\Delta_A(h) - \Delta_{A'}(h)$ has the same conditional distribution as $\mathcal F^+(h) - \mathcal F^+(h^{A\oplus A'}) = \Delta_{A\oplus A'}(h)$.
Thus, we get \eqref{eq-Delta-two-point} from \eqref{eq-Delta-one-point}.
\end{proof}

\subsection{Proof of Theorem~\ref{RFIM3}}

By \eqref{def_mu}, under the law $\mathbb{Q}^+$, the joint density $\nu^+$ of $h$ and $\sigma$ is given by
\begin{equation}
\label{def_p}
\nu^+(h , \sigma)=\prod_{u \in \Lambda_N} \left(\frac{1}{\sqrt{2\pi}}e^{-\frac{1}{2}h_u^2} \right) \times \frac{e^{-\frac{1}{T} H^{+, \Lambda_N, \epsilon h}(\sigma)}}{\mathcal{Z}^+(h)}.
\end{equation}
Note that the density associated with $h$ is a continuous density and the one associated with $\sigma$ is a discrete probability mass.

We now set up the framework for our Peierls argument. Let $\mathfrak A$ be the collection of all simply connected subsets $A\subset \Lambda_N$ with $o\in A$. We wish to bound $|\{A \in \mathfrak A: |\partial A| = n\}|$.  We say two edges $(u_1, u_2)$ and $(v_1, v_2)$ are neighboring each other if $u_i$ and $v_j$ has $\ell_1$-distance at most 2 for some $i, j\in \{1, 2\}$. With this neighboring relation, we see that $\partial A$ forms a connected component and as a result we get that
\begin{equation}\label{eq-mathfrak-A-size}
|\{A \in \mathfrak A: |\partial A| = n\}| \leq (2dn)^d (16d^3)^{2n}\,,
\end{equation}
(note that the bound above does not depend on $N$).
This is because the right hand side above corresponds to an upper bound on the number of trees (on the edges of $\mathbb Z^d$) of $n$ edges with a starting edge in $\Lambda_n$, and such trees can be enumerated by a depth-first search process---the search process has $2n$ steps (since it visits each edge twice), where we have $(2dn)^d$ choices for the initial step and each later step has at most $16d^3$ choices. (One may also see \cite{Ruelle69, LM98} for sharper estimates on $|\mathfrak A|$ although these are insignificant for us.) For $\sigma\in \{-1, 1\}^{\Lambda_N}$, if $\sigma_o = -1$ we let $\mathcal A_\sigma$ be the simply connected component that contains all vertices either in the sign component of the origin or enclosed by the sign component (i.e., $\mathcal A_\sigma$ is the simply connected component enclosed by the outmost boundary of the sign component at the origin), and if $\sigma_o = 1$ we let $\mathcal A_\sigma = \emptyset$. Note that when $\sigma_o = -1$, we have $\mathcal A_\sigma \subset \Lambda_N$ (since the boundary condition is plus and we make the convention that $\sigma_u = 1$ if $u \not \in \Lambda_N$) and thus $\mathcal A_\sigma \in \mathfrak A$. Our mental picture is that for our generalized Peierls argument, we consider the mapping $(h, \sigma) \mapsto (h^{\mathcal A_{\sigma}}, \sigma^{\mathcal A_\sigma})$ (so in particular the mapping flips the spin at the origin if it is initially a minus spin). For any $(u,v) \in \partial \mathcal{A}_\sigma$, we have $\sigma_u \sigma_v = -1$ and $\sigma_u^{\mathcal A_\sigma} \sigma_v^{\mathcal A_\sigma} = 1$. Thus, by \eqref{def_h} and \eqref{def_p},
\begin{equation}
\label{key}
\frac{\nu^+(h,\sigma)}{\nu^+(h^{\mathcal A_\sigma},\sigma^{\mathcal A_\sigma})}=e^{- \frac{2}{T} |\partial A|} \frac{\mathcal{Z}^+(h^{\mathcal A_\sigma})}{\mathcal{Z}^+(h)}\,.
\end{equation}
The first term on the right hand side above decays exponentially in the boundary size. We now bound the second term. By Lemma~\ref{lem-upper-bound-Delta} (also recall \eqref{eq-Delta-expectation}), we can adapt the multiscale analysis in \cite[Page 866--Page 869]{FFS84} or \cite[Page 116--Page 118]{Bovier06} and show that
\begin{equation}
\label{ising_bound_2}
\mathbb{P}\Big(\sup_{A\in \mathfrak A} \frac{|\Delta_A(h)|}{|\partial A|} \geq 1\Big) \leq \exp (-c/\epsilon^2)\,,
\end{equation}
where $c>0$ is a small constant and $\epsilon \leq c$.
The proof in \cite{FFS84} is based on the coarse graining method, which roughly speaking considers a simply connected set as a disjoint union of ``connected'' dyadic boxes of varying sizes.
Note that although the definition of $\Delta_A(h)$ is different from $F_A(h)$ in \cite{FFS84}, Lemma~\ref{lem-upper-bound-Delta} shows that $\Delta_A(h)/4$ satisfies the condition \cite[Eq.(10)]{FFS84} which is the only property of $F_A(h)$ required to prove \eqref{ising_bound_2} for small $\epsilon$. As a result, the adaption of the proof is essentially verbatim (note that the proof in \cite{FFS84} was only written for $d=3$ but this can be easily extended to all $d\geq 3$ as noted in \cite{Bovier06}).
From \eqref{ising_bound_2} (also recall \eqref{eq-def-Delta}) we immediately get that for $\epsilon \leq c$
\begin{equation}
\label{ising_bound}
\mathbb{P}(\mathcal E^c) \leq \exp (-c/\epsilon^2) \mbox{ where } \mathcal E = \Big\{ \sup_{A\in \mathfrak A} \frac{\mathcal{Z}^+(h^A)}{\mathcal{Z}^+(h)} \leq e^{|\partial A|/T} \Big\}\,.
\end{equation}

We are now ready to provide the proof for Theorem~\ref{RFIM3}. Note that $\mathcal E$ is an event measurable with respect to $\{h_v: v\in \Lambda_N\}$. We have
\begin{align*}
\mathbb{Q}^+(\sigma_o=-1) &\leq \mathbb{Q}^+(\{\sigma_o=-1 \} \cap \mathcal{E}) + \mathbb{Q}^+(\mathcal{E}^c) \\
& \leq \sum_{A\in \mathfrak A} \int_{h\in \mathcal E}\sum_{\sigma: \mathcal A_\sigma = A} \nu^+ (h, \sigma) dh +\mathbb{P}(\mathcal{E}^c) \qquad \mbox{(since $\mathcal A_\sigma\in \mathfrak A$ if $\sigma_o = -1$)} \\
&\leq \sum_{A \in \mathfrak A} \frac{\int_{h \in \mathcal E}\sum_{\sigma: \mathcal A_\sigma = A} \nu^+(h,\sigma) dh}{\int_{h\in \mathcal E} \sum_{\sigma: \mathcal A_\sigma = A} \nu^+(h^A,\sigma^A) dh} + \mathbb{P}(\mathcal E^c) \qquad \mbox{ (since $\int_{h} \sum_\sigma \nu^+(h^A,\sigma^A) dh = 1$)}\,,
\end{align*}
where the point of the last step is to compare probabilities along the mapping $(h, \sigma) \mapsto (h^{\mathcal A_\sigma}, \sigma^{\mathcal A_\sigma})$. Continuing the analysis, we obtain that for $T < c$
\begin{align}
\mathbb{Q}^+(\sigma_o=-1)&\leq \sum_{A\in \mathfrak A}\sup_{h \in \mathcal{E}, \mathcal A_\sigma = A} \frac{\nu^+(h, \sigma)}{\nu^+(h^A,\sigma^A)}  +\mathbb{P}(\mathcal{E}^c) \nonumber\\
&= \sum_{A\in \mathfrak A} \sup_{h\in \mathcal E}\Big( e^{-\frac{2}{T}|\partial A|} \frac{\mathcal{Z}^+(h^A)}{\mathcal{Z}^+(h)} \Big) +\mathbb{P}(\mathcal{E}^c)  \qquad \mbox{ (by \eqref{key})}\nonumber \\
&\leq \sum_{n \geq 1} \sum_{A\in \mathfrak A, |\partial A|=n} e^{-\frac{2}{T} |\partial A|+ \frac{1}{T} |\partial A|} +\mathbb{P}(\mathcal{E}^c) \qquad \mbox{ (by \eqref{ising_bound})} \nonumber\\
&\leq \sum_{n \geq 1} (2dn)^d (16d^3)^{2n} e^{-  n/T} +\mathbb{P}(\mathcal{E}^c)  \qquad \mbox{ (by \eqref{eq-mathfrak-A-size})} \nonumber\\
&\leq e^{-c/T}+ e^{-c/ \epsilon^2}   \qquad \mbox{(by \eqref{ising_bound} and that $T$ is assumed to be small)}\label{eq-RFIM3}
\end{align}
where in the last step we have decreased the value of $c$. Now a simple application of Markov's inequality completes the proof of Theorem~\ref{RFIM3} (and we adjust the value of $c$ again).

\subsection{Proof of Theorem~\ref{RFIM2}}\label{sec:RFIM2d}
For any $A\subset \mathbb Z^d$, let $H_A= \sum_{u \in A} h_u$. We assume that $N \leq e^{\frac{c\epsilon^{-4/3}}{ \log(1/\epsilon)}}$ where $c$ is a small constant to be chosen.
For $d=2$, in order to bound $\sup_{A\in \mathfrak A} \frac{\Delta_A(h)}{|\partial A|}$, we still wish to use Lemma~\ref{lem-upper-bound-Delta}. However, for $d=2$ it is substantially more complicated and the bound on the corresponding Gaussian process $\{\frac{H_A}{|\partial A|}: A\in \mathfrak A\}$ was obtained in \cite{DW20} via a sophisticated multi-scale analysis argument, which employs Gaussian properties in a significant way (and thus it is difficult to just adapt the proof). As a result, in this case, we will apply Lemma~\ref{lem-Talagrand}. One subtlety is that our process $\{\frac{\Delta_A(h)}{|\partial A|}: A\in \mathfrak A\}$ is normalized by the boundary size. To address this, we partition all the sets into collections of sets with comparable boundary sizes. For $k \geq 0$, we write $\mathcal{B}_k = \{ A\in \mathfrak A: |\partial A| \in [2^k, 2^{k+1}]\}$.  Then, by \cite[Theorem 1.4]{DW20}
$$
\mathbb{E}\Big(\sup_{A \in \mathcal{B}_k} H_A\Big) \leq 2^{k+1} \mathbb{E}\Big(\max\Big\{\sup_{A \in \mathcal{B}_k} \frac{H_A}{|\partial A|}, 0\Big\}\Big) \leq C 2^k (\log N \log \log N)^{3/4}\,,
$$
where $C>0$ is a constant.
Note that $\E(H_A - H_{A'})^2 = |A \oplus A'|$. For notation convenience, we let $\mathbb P_k$ and $\mathbb E_k$ be the measure and the expectation respectively with respect to $\P(\cdot \mid \{h_v: v\in \Lambda_{2^{k+1}}^c\})$.
 By Lemmas~\ref{lem-upper-bound-Delta} and \ref{lem-Talagrand} (also recall \eqref{eq-Delta-expectation}), we get
$$
\mathbb{E}_k \Big(\sup_{A \in \mathcal{B}_k} \Delta_A(h) \Big) \leq C \mathbb{E} \Big( \sup_{A \in \mathcal{B}_k} 2\epsilon H_A \Big) \leq C \epsilon 2^k (\log N \log \log N)^{3/4} \,,
$$where $C>0$ is a constant whose value may be increased from the last occurrence.
Therefore, by the assumption that $N \leq e^{\frac{c\epsilon^{-4/3}}{ \log(1/\epsilon)}}$ we can choose a sufficiently small constant $c>0$ such that
$$
\mathbb{E}_k \Big(\sup_{A \in \mathcal{B}_k} \Delta_A(h) \Big) \leq 2^{k-1}.
$$
By \eqref{eq-Delta-Lip}, we see that conditioned on the values of $\{h_v: v\in \Lambda_{2^{k+1}}^c$\}, for each $A\in \mathcal B_k$ we have that $\Delta_A(h)$ is a $(2\epsilon)$-Lipschitz function of $\{h_v: v\in \Lambda_{2^{k+1}}\}$ and thus so is $\sup_{A\in \mathcal B_k}\Delta_A(h)$. By the Gaussian concentration inequality (see \cite{Borell75, SudakovTsirelson74}, and see also \cite[Theorem 2.1]{Adler90} and \cite[Theorem 3.25]{Van16})
$$
\mathbb{P}_k\Big(\sup_{A \in \mathcal{B}_k} \frac{\Delta_A(h)}{|\partial A|} \geq 1\Big) \leq \mathbb{P}_k\Big(\sup_{A \in \mathcal{B}_k} \Delta_A(h) \geq 2^k \Big) \leq \exp(-c/{\epsilon^2}).
$$
Taking expectation over the preceding inequality and summing over $k$ (noting $\mathcal{B}_k =\emptyset$ for $k \geq 10 \log N = O(\epsilon^{-4/3})$), we have for $\epsilon < c$
\begin{align*}
\mathbb{P}(\mathcal E^c) \leq \exp(-c/\epsilon^2), \quad \mbox{ where }  \mathcal E = \Big\{\sup_{A\in \mathfrak A} \frac{\Delta_A(h)}{|\partial A|} \leq 1 \Big\} \overset{\eqref{eq-def-Delta}}{=} \Big\{ \sup_{A\in \mathfrak A} \frac{\mathcal{Z}^+(h^A)}{\mathcal{Z}^+(h)} \leq e^{|\partial A|/T} \Big\}\,.
\end{align*}
(Note that we have decreased the value of $c$ from the last occurrence.)

We are now ready to prove Theorem \ref{RFIM2}.
Fix $\mathsf m \in (0,1)$. For $\epsilon, T\leq c$ (recall that $N \leq e^{\frac{c\epsilon^{-4/3}}{ \log(1/\epsilon)}}$), we can repeat the derivation of \eqref{eq-RFIM3} verbatim and obtain that
$$
\mathbb{Q}^\pm[\sigma_o=\mp] \leq e^{-c/T} +  e^{-c/\epsilon^2}\,.
$$
Thus, $m(T, \Lambda_N, \epsilon)  \geq \mathsf m$ for $T, \epsilon$ small enough. This completes the proof of the theorem.

\section{Long range order for random field Potts model}

In this section we prove Theorem~\ref{RFPM3}. We first make some definitions similar to those at the beginning of Section~\ref{sec:RFIM} with appropriate modifications. Fix $d, \mathsf q\geq 3$ and assume that $T, \epsilon \leq c$ for a small constant $c>0$ (to be chosen). We refer $\mathbb{P}$ and $\mathbb E$ to the measure and the expectation respectively with respect to the external field $\{h_{\mathsf k, v}: v\in \mathbb Z^d, \mathsf 1\leq \mathsf k \leq \mathsf q\}$. Similar to the case for RFIM, we let $\mathbb{Q}_{T, \Lambda_N, \epsilon}^{\mathsf k}$ be the joint measure for $(h, \sigma)$, so that for any $A\subset (\mathbb R^{\mathsf q})^{\Lambda_N}$ and $B\subset \{\mathsf 1, \ldots, \mathsf q\}^{\Lambda_N}$
$$\mathbb{Q}_{T, \Lambda_N, \epsilon}^{\mathsf k} (h\in A, \sigma\in B) =  \int_A \mu_{T, \Lambda_N, \epsilon h}^{\mathsf k}(B) d \mathbb P(h)\,.$$
Since $T, \epsilon, N$ will be fixed throughout the proof, (similarly to Section~\ref{sec:RFIM}) we will drop them from the superscripts/subscripts for notation convenience.

By symmetry, in order to prove Theorem~\ref{RFPM3} it suffices to prove
\begin{equation}\label{eq-RFPM}
\mathbb Q^{\mathsf 1}(\sigma_o = \mathsf 2) \leq  e^{-c/T} +  e^{-c/\epsilon^2}\,,
\end{equation}
since this immediately implies an upper bound on $\mathbb Q^{\mathsf 1}(\sigma_o \neq \mathsf 1)$ and thus an application of Markov's inequality yields Theorem~\ref{RFPM3} (where we adjust the value of $c$).

We now turn to the proof of \eqref{eq-RFPM}. We define a rotation $\theta$ on $\{\mathsf 1, \ldots, \mathsf q\}$ by $\theta(\mathsf k) = \mathsf k- \mathsf 1$ for $\mathsf k \in \{\mathsf 2, \ldots, \mathsf q\}$ and $\theta(\mathsf 1) = \mathsf q$ (so $\theta$ plays a similar role as flipping the spin for the Ising model). For any $\mathsf 1\leq \mathsf j \leq \mathsf q$ and any $A\subset \mathbb Z^d$, define
$$h^{A, \mathsf j}_{\mathsf k, v} = \begin{cases}
h_{\mathsf k, v},  \mbox{ if } v\in A^c,\\
h_{\theta^{-\mathsf j}(\mathsf k), v}, \mbox{ if } v\in A;
\end{cases}
\mbox{ and }
\sigma^{A, \mathsf j}_v =  \begin{cases}
\sigma_v, \mbox{ if } v\in A^c,\\
\theta^{\mathsf j}(\sigma_v), \mbox{ if } v\in A\,.
\end{cases}
$$  We further define  the free energy difference by
\begin{equation}\label{eq-def-Delta-Potts}
\Delta_{A, \mathsf j}(h) = \mathcal F^{\mathsf 1}(h) - \mathcal F^{\mathsf 1}(h^{A, \mathsf j}) \overset{\eqref{eq-def-free-energy}}{=} -T \log \mathcal Z^{\mathsf 1}(h) + T \log \mathcal Z^{\mathsf 1}(h^{A, \mathsf j})\,.
\end{equation}
We first prove an analogue of Lemma~\ref{lem-upper-bound-Delta}
\begin{lemma}
\label{lem-upper-bound-Delta-RFPM}
For any $A, A'\subset \Lambda_N$, $\mathsf 1\leq \mathsf j\leq \mathsf q$ and $\lambda > 0$, we have
\begin{align}
\P(|\Delta_{A, \mathsf j}(h)| \geq \lambda \mid \{h_{\mathsf k, v}: \mathsf 1\leq \mathsf k\leq \mathsf q, v\in A^c\}) &\leq 2 e^{-\frac{\lambda^2}{8 \epsilon^2 \mathsf q |A|}}\,, \label{eq-Delta-one-point-RFPM}\\
\mathbb P(|\Delta_{A, \mathsf j}(h) - \Delta_{A', \mathsf j}(h)| \geq \lambda \mid \{h_{\mathsf k, v}: \mathsf 1\leq \mathsf k\leq \mathsf q, v\in (A \cup A')^c\}) &\leq 2 e^{-\frac{\lambda^2}{8 \epsilon^2 \mathsf q |A\oplus A'|}}\,, \label{eq-Delta-two-point-RFPM}
\end{align}
where $A \oplus A'$ is the symmetric difference between $A$ and $A'$.
\end{lemma}
\begin{proof}
We first prove \eqref{eq-Delta-one-point-RFPM}. Fix $\mathsf j$.
By symmetry (more precisely the identical law of $h_{\mathsf k, v}$ for $\mathsf 1\leq \mathsf k\leq \mathsf q$ and $v\in \mathbb Z^d$),
\begin{equation}\label{eq-Delta-expectation-Potts}
\mathbb{E} (\Delta_{A, \mathsf j}(h) \mid \{h_{\mathsf k, v}: \mathsf 1\leq \mathsf k\leq \mathsf q, v\in A^c\}) =0\,.
\end{equation} Furthermore, for any $v\in A$ and $\mathsf 1\leq \mathsf i\leq \mathsf q$, by \eqref{def_z_potts} and \eqref{def_h_potts}
\begin{equation}\label{eq-Delta-Lip-Potts}
\begin{split}
\Big| \frac{\partial}{\partial_{h_{\mathsf i, v}}} \Delta_{A, \mathsf j}(h) \Big|&= \Big|- \frac{\sum_{\sigma } \epsilon   \mathbf {1}_{\sigma_{v} = \mathsf i} e^{-\frac{1}{T} H^{\mathsf 1,\Lambda_N , \epsilon h}(\sigma)}}{\mathcal{Z}^{\mathsf 1}(h)} + \frac{\sum_{\sigma }  \epsilon  \mathbf{1}_{\theta^{- \mathsf j}(\sigma_{v}) = \mathsf i} e^{-\frac{1}{T} H^{\mathsf 1, \Lambda_N, \epsilon h^{A, \mathsf j}}(\sigma)} }{\mathcal{Z}^{\mathsf 1}(h^{A, \mathsf j})}  \Big|\\
&=\epsilon \Big| - \mu^{\mathsf 1}_{T, \Lambda_N,\epsilon h}(\sigma_v = \mathsf i) + \mu^{\mathsf 1}_{T, \Lambda_N, \epsilon h^{A, \mathsf j}}( \theta^{-\mathsf j}(\sigma_v) =\mathsf i) \Big| \leq  2 \epsilon.
\end{split}
\end{equation}
Now, by the Gaussian concentration inequality (see \cite{Borell75, SudakovTsirelson74}, and see also \cite[Theorem 2.1]{Adler90} and \cite[Theorem 3.25]{Van16}), we obtain \eqref{eq-Delta-one-point-RFPM}.

We next prove \eqref{eq-Delta-two-point-RFPM}. In the rest of the proof, we consider probability measures conditioned on $\{h_{\mathsf k, v}: \mathsf 1\leq \mathsf k\leq \mathsf q, v\in (A \cup A')^c\}$.
Since $\Delta_{A, \mathsf j}(h) - \Delta_{A', \mathsf j}(h) = \mathcal F^{\mathsf 1}(h^{A', \mathsf j}) - \mathcal F^{\mathsf 1}(h^{A, \mathsf j})$ and since $(h^{A, \mathsf j}, h^{A', \mathsf j})$ has the same conditional distribution as $(h^{A\oplus A', \mathsf j}, h)$, we obtain that
$\Delta_{A, \mathsf j}(h) - \Delta_{A', \mathsf j}(h)$ has the same conditional distribution as $\mathcal F^{\mathsf 1}(h) - \mathcal F^{\mathsf 1}(h^{A\oplus A', \mathsf j}) = \Delta_{A\oplus A', \mathsf j}(h)$.
Thus, we get \eqref{eq-Delta-two-point-RFPM} from \eqref{eq-Delta-one-point-RFPM}.
\end{proof}
We next prove a bound on the supremum for the free energy difference.
\begin{lemma}\label{lem-Delta-bound-RFPM}
There exists a constant $c = c(d, \mathsf q)> 0$ such that for $0 \leq \epsilon<c$ we have
$$\mathbb P(\mathcal E^c) \leq e^{-c/\epsilon^2} \mbox{ where }  \mathcal E = \Big\{\sup_{\mathsf 1\leq \mathsf j\leq \mathsf q, A \in \mathfrak A} \frac{\Delta_{A, \mathsf j}}{|\partial A|} \leq \frac{1}{2\mathsf q}\Big\} \overset{\eqref{eq-def-Delta-Potts}}{=} \Big\{\sup_{\mathsf 1\leq \mathsf j\leq \mathsf q, A \in \mathfrak A} \frac{\mathcal Z^{\mathsf 1}(h^{A, \mathsf j})}{\mathcal Z^{\mathsf 1}(h)} \leq e^{\frac{| \partial A|}{2\mathsf q T}}\Big\}\,.$$
\end{lemma}
\begin{proof}
For a fixed $\mathsf j$, by Lemma~\ref{lem-upper-bound-Delta-RFPM} (also recall \eqref{eq-Delta-expectation-Potts}) and the same derivation for \eqref{ising_bound_2} we get that for $\epsilon < c$ (and a small constant $c>0$)
$$\P\Big(\sup_{A \in \mathfrak A} \frac{\Delta_{A, \mathsf j}}{|\partial A|} \geq \frac{1}{2\mathsf q}\Big) \leq e^{-c/\epsilon^2}\,.$$
Taking a union bound over $\mathsf j \in \{\mathsf 1, \ldots, \mathsf q\}$ yields the lemma (possibly we will need to decrease the value of $c$).
\end{proof}

\begin{proof}[Proof of \eqref{eq-RFPM}]
 For $\sigma\in \{\mathsf 1, \ldots, \mathsf q\}^{\Lambda_N}$, if $\sigma_o  = \mathsf 2$ we let $\mathcal A_\sigma$ be the simply connected component that contains all vertices either on or enclosed by the component of spin $\mathsf 2$ containing the origin (i.e., $\mathcal A_\sigma$ is the simply connected component enclosed by the outmost boundary of the connected component sharing the spin with the origin), and if $\sigma_o \neq \mathsf 2$ we let $\mathcal A_\sigma = \emptyset$. Our mental picture is that for our generalized Peierls argument, we consider a collection of mappings $(h, \sigma) \mapsto (h^{\mathcal A_{\sigma}, \mathsf j}, \sigma^{\mathcal A_\sigma, \mathsf j})$ for $\mathsf j = \mathsf 1, \ldots, \mathsf q-  \mathsf 1$.

By \eqref{def_mu_potts}, under the law $\mathbb{Q}^{\mathsf 1}$, the joint density $\nu^{\mathsf 1}$ of $h$ and $\sigma$ is given by
\begin{equation*}
\nu^{\mathsf 1} (h,\sigma)=\prod_{u \in \Lambda_N} \prod_{\mathsf i= \mathsf 1}^{\mathsf q} \left( \frac{1}{\sqrt{2 \pi}} e^{-\frac{1}{2}{(h_{\mathsf i, u})^2}} \right) \times \frac{e^{-\frac{1}{T} H^{\mathsf 1}(\sigma, \Lambda_N, \epsilon h)}}{\mathcal{Z}^{\mathsf 1}(h)}\,.
\end{equation*}
For $\sigma$ with $\sigma_o = \mathsf 2$, we have that $\mathcal A_\sigma\in \mathfrak A$ and $\sigma_u\neq \sigma_v$ for any $(u, v)\in \partial \mathcal A_\sigma$ where by the $\mathsf 1$ boundary condition we make the convention that $\sigma_u = \mathsf 1$ if $u \not \in \Lambda_N$. Thus,
by \eqref{def_h_potts} we get that for any $\mathsf j$
\begin{equation*}
\nu^{\mathsf 1}(h,\sigma)=\exp \Big(- \frac{1}{T} \sum_{(u, v)\in \partial \mathcal A_\sigma} \mathbf{1}_{\sigma^{\mathcal A_\sigma, \mathsf j}_u =\sigma^{\mathcal A_\sigma, \mathsf j}_v} \Big) \frac{\mathcal{Z}^{\mathsf 1}(h^{\mathcal A_\sigma, \mathsf j})}{\mathcal{Z}^{\mathsf 1}(h)} \nu^{\mathsf 1}(h^{\mathcal A_\sigma, \mathsf j},\sigma^{\mathcal A_\sigma, \mathsf j})\,.
\end{equation*}
Combined with the fact that $
\sum_{\mathsf j= \mathsf 1}^{\mathsf q- \mathsf 1} \sum_{(u, v)\in \partial \mathcal A_\sigma} \mathbf{1}_{\sigma^{\mathcal A_\sigma, \mathsf j}_u =\sigma^{\mathcal A_\sigma, \mathsf j}_v} =\sum_{\mathsf j= \mathsf 1}^{\mathsf q- \mathsf 1} \sum_{(u, v)\in \partial \mathcal A_\sigma} \mathbf{1}_{\sigma^{\mathcal A_\sigma, \mathsf j}_u =\sigma_v}
=|\partial \mathcal A_\sigma|
$, it yields the following analogue of \eqref{key}:
\begin{equation}
\label{potts_key}
\frac{\nu^{\mathsf 1}(h,\sigma)}{\sum_{\mathsf j= \mathsf 1}^{\mathsf q- \mathsf 1} \nu^{\mathsf 1}(h^{\mathcal A_\sigma, \mathsf j},\sigma^{\mathcal A_\sigma, \mathsf j})} \leq e^{-\frac{|\partial \mathcal A_\sigma|}{(\mathsf q-  \mathsf 1) T}} \sup_{\mathsf 1 \leq \mathsf j \leq \mathsf q- \mathsf 1} \left( \frac{\mathcal{Z}^{\mathsf 1}(h^{\mathcal A_\sigma, \mathsf j})}{\mathcal{Z}^{\mathsf 1}(h)} \right) .
\end{equation}
At this point, by a similar derivation of \eqref{eq-RFIM3}, we get that (since the proof is highly similar, we will omit some computational details) for $T<c$ (we will decrease the value of $c$ if necessary)
\begin{align*}
\mathbb{Q}^{\mathsf 1}(\sigma_o = \mathsf 2) &\leq \mathbb{Q}^{\mathsf 1}(\{\sigma_o = \mathsf 2 \} \cap \mathcal{E}) + \mathbb{Q}^{\mathsf 1}(\mathcal{E}^c) \\
&\leq \sum_{A\in \mathfrak A} \int_{h \in \mathcal E} \sum_{\sigma: \mathcal A_\sigma = A} \nu^{\mathsf 1} (h, \sigma)dh  +\mathbb{P}(\mathcal{E}^c) \qquad \mbox{(since $\mathcal A_\sigma\in \mathfrak A$ if $\sigma_o = \mathsf 2$)}  \\
&\leq \sum_{A\in \mathfrak A} \frac{(\mathsf{q} -\mathsf 1) \int_{h \in \mathcal E} \sum_{\sigma: \mathcal A_\sigma = A} \nu^{\mathsf 1} (h, \sigma) dh}{ \sum_{\mathsf j=\mathsf 1}^{\mathsf q- \mathsf 1}\int_{h \in \mathcal E } \sum_{\sigma: \mathcal A_\sigma = A}  \nu^{\mathsf 1}(h^{A,\mathsf j}, \sigma^{A, \mathsf j}) dh}+ \mathbb{P}(\mathcal{E}^c) \quad \mbox{ (since $\int_{h} \sum_{\sigma}  \nu^{\mathsf 1}(h^{A, \mathsf j}, \sigma^{A, \mathsf j})dh  = 1$)}\\
&\leq \sum_{A\in \mathfrak A}  (\mathsf{q}- \mathsf 1) e^{-\frac{|\partial A|}{2(\mathsf q- \mathsf 1) T}} +\mathbb{P}(\mathcal{E}^c)  \qquad \mbox{ (by \eqref{potts_key} and by definition of $\mathcal E$)}\\
&\leq  e^{-c/T}+  e^{-c/\epsilon^2}\qquad \mbox{ (by \eqref{eq-mathfrak-A-size} and Lemma~\ref{lem-Delta-bound-RFPM})}\,,
\end{align*}
as required.
\end{proof}

\medskip

\noindent {\bf Acknowledgement.} We thank Jianping Jiang, Jian Song and Rongfeng Sun for a careful reading of an earlier version of the manuscript and for their helpful comments on improving exposition.


\small

\begin{thebibliography}{10}

\bibitem{Adler90}
R.~J. Adler.
\newblock {\em An introduction to continuity, extrema, and related topics for
  general {G}aussian processes}, volume~12 of {\em Institute of Mathematical
  Statistics Lecture Notes---Monograph Series}.
\newblock Institute of Mathematical Statistics, Hayward, CA, 1990.

\bibitem{AHP20}
M.~Aizenman, M.~Harel, and R.~Peled.
\newblock Exponential decay of correlations in the 2{D} random field {I}sing
  model.
\newblock {\em J. Stat. Phys.}, 180(1-6):304--331, 2020.

\bibitem{AP19}
M.~Aizenman and R.~Peled.
\newblock A power-law upper bound on the correlations in the {$2D$} random
  field {I}sing model.
\newblock {\em Comm. Math. Phys.}, 372(3):865--892, 2019.

\bibitem{AW90}
M.~Aizenman and J.~Wehr.
\newblock Rounding effects of quenched randomness on first-order phase
  transitions.
\newblock {\em Comm. Math. Phys.}, 130(3):489--528, 1990.

\bibitem{Ber85}
A.~Berretti.
\newblock Some properties of random {I}sing models.
\newblock {\em J. Statist. Phys.}, 38(3-4):483--496, 1985.

\bibitem{Borell75}
C.~Borell.
\newblock The {B}runn-{M}inkowski inequality in {G}auss space.
\newblock {\em Invent. Math.}, 30(2):207--216, 1975.

\bibitem{Bovier06}
A.~Bovier.
\newblock {\em Statistical mechanics of disordered systems}, volume~18 of {\em
  Cambridge Series in Statistical and Probabilistic Mathematics}.
\newblock Cambridge University Press, Cambridge, 2006.
\newblock A mathematical perspective.

\bibitem{BK88}
J.~Bricmont and A.~Kupiainen.
\newblock Phase transition in the $3$d random field {I}sing model.
\newblock {\em Comm. Math. Phys.}, 116(4):539--572, 1988.

\bibitem{CJN18}
F.~Camia, J.~Jiang, and C.~M. Newman.
\newblock A note on exponential decay in the random field {I}sing model.
\newblock {\em J. Stat. Phys.}, 173(2):268--284, 2018.

\bibitem{Chalker83}
J.~Chalker.
\newblock On the lower critical dimensionality of the ising model in a random
  field.
\newblock {\em J. Phys. C}, 16(34):6615--6622, 1983.

\bibitem{Chatterjee18}
S.~Chatterjee.
\newblock On the decay of correlations in the random field {I}sing model.
\newblock {\em Comm. Math. Phys.}, 362(1):253--267, 2018.

\bibitem{DHP21}
P.~Dario, M.~Harel, and R.~Peled.
\newblock Quantitative disorder effects in low-dimensional spin systems.
\newblock Preprint, arXiv:2101.01711.

\bibitem{DSS21}
J.~Ding, J.~Song, and R.~Sun.
\newblock A new correlation inequality for ising models with external fields.
\newblock Preprint, arXiv:2107.09243.

\bibitem{DW20}
J.~Ding and M.~Wirth.
\newblock Correlation length of two-dimensional random field ising model via
  greedy lattice animal.
\newblock Preprint, arXiv:2011.08768.

\bibitem{DX21}
J.~Ding and J.~Xia.
\newblock Exponential decay of correlations in the two-dimensional random field
  {I}sing model.
\newblock {\em Invent. Math.}, 224(3):999--1045, 2021.

\bibitem{Dudley67}
R.~M. Dudley.
\newblock The sizes of compact subsets of {H}ilbert space and continuity of
  {G}aussian processes.
\newblock {\em J. Functional Analysis}, 1:290--330, 1967.

\bibitem{Fernique71}
X.~Fernique.
\newblock R\'{e}gularit\'{e} de processus gaussiens.
\newblock {\em Invent. Math.}, 12:304--320, 1971.

\bibitem{FFS84}
D.~S. Fisher, J.~Fr\"{o}hlich, and T.~Spencer.
\newblock The {I}sing model in a random magnetic field.
\newblock {\em J. Statist. Phys.}, 34(5-6):863--870, 1984.

\bibitem{FI84}
J.~Fr\"{o}hlich and J.~Z. Imbrie.
\newblock Improved perturbation expansion for disordered systems: beating
  {G}riffiths singularities.
\newblock {\em Comm. Math. Phys.}, 96(2):145--180, 1984.

\bibitem{Imbrie85}
J.~Z. Imbrie.
\newblock The ground state of the three-dimensional random-field {I}sing model.
\newblock {\em Comm. Math. Phys.}, 98(2):145--176, 1985.

\bibitem{IM75}
Y.~Imry and S.-K. Ma.
\newblock Random-field instability of the ordered state of continuous symmetry.
\newblock {\em Phys. Rev. Lett.}, 35:1399--1401, Nov 1975.

\bibitem{LM98}
J.~L. Lebowitz and A.~E. Mazel.
\newblock Improved {P}eierls argument for high-dimensional {I}sing models.
\newblock {\em J. Statist. Phys.}, 90(3-4):1051--1059, 1998.

\bibitem{Peierls1936}
R.~Peierls.
\newblock On ising's model of ferromagnetism.
\newblock {\em Mathematical Proceedings of the Cambridge Philosophical
  Society}, 32(3):477–481, 1936.

\bibitem{Ruelle69}
D.~Ruelle.
\newblock {\em Statistical mechanics: {R}igorous results}.
\newblock W. A. Benjamin, Inc., New York-Amsterdam, 1969.

\bibitem{SudakovTsirelson74}
V.~N. Sudakov and B.~S. Tsirel'son.
\newblock Extremal properties of half-spaces for spherically invariant
  measures.
\newblock {\em J. Sov. Math}, 9:9--18, 1978.

\bibitem{Talagrand87}
M.~Talagrand.
\newblock Regularity of {G}aussian processes.
\newblock {\em Acta Math.}, 159(1-2):99--149, 1987.

\bibitem{Talagrand96}
M.~Talagrand.
\newblock Majorizing measures: the generic chaining.
\newblock {\em Ann. Probab.}, 24(3):1049--1103, 1996.

\bibitem{Talagrand05}
M.~Talagrand.
\newblock {\em The generic chaining}.
\newblock Springer Monographs in Mathematics. Springer-Verlag, Berlin, 2005.
\newblock Upper and lower bounds of stochastic processes.

\bibitem{Van16}
R.~van Handel.
\newblock {\em Probability in High Dimension}.
\newblock Lecture notes in progress, available at
  \textit{https://web.math.princeton.edu/~rvan/APC550.pdf}.

\end{thebibliography}
\end{document}